\documentclass{amsart}
\usepackage{amsfonts}
\usepackage{amsmath,amssymb}
\usepackage{amsthm}
\usepackage{amscd}
\usepackage{graphics}
\usepackage{graphicx}

\theoremstyle{remark}{
\newtheorem{Def}{{\rm Definition}}
\newtheorem{Ex}{{\rm Example}}
\newtheorem{Rem}{{\rm Remark}}
\newtheorem{Prob}{{\rm Problem}}
\newtheorem*{MainProb}{Main Problem}
}
\theoremstyle{plain}{

\newtheorem{Prop}{Proposition}
\newtheorem{Thm}{Theorem}
\newtheorem{MainThm}{Main Theorem}

\newtheorem{Fact}{Fact}
}

\begin{document}
\title[Global properties of Reeb spaces with non-zero top homology groups]{
Global topologies of Reeb spaces of stable fold maps with non-trivial top homology groups}
\author{Naoki Kitazawa}
\address{Institute of Mathematics for Industry, Kyushu University, 744 Motooka, Nishi-ku Fukuoka 819-0395, Japan\\
 TEL (Office): +81-92-802-4402 \\
 FAX (Office): +81-92-802-4405 \\
}
\email{n-kitazawa@imi.kyushu-u.ac.jp}
\urladdr{https://naokikitazawa.github.io/NaokiKitazawa.html}
\maketitle
\begin{abstract}
The {\it Reeb space} of a continuous map is the space of all (elements representing) connected components of preimages endowed with the quotient topology induced from the natural equivalence relation on the domain.
These objects are strong tools in (differential) topological theory of Morse functions, {\it fold} maps, which are their higher dimensional variants, and so on: they are in general polyhedra whose dimensions are same as those of the targets.
In suitable cases Reeb spaces inherit topological information such as homology groups, cohomology rings, and so on, of the manifolds. 

This presents the following problem: what are global topologies of Reeb spaces of these smooth maps of suitable classes like? The present paper presents families of stable fold maps having Reeb spaces with non-trivial top homology groups with
their (co)homology groups (and rings).
Related studies on the global topologies from the viewpoints of the singularity theory of differentiable maps and differential topology have been presented by various researchers including the author. The author previously constructed families of fold maps with Reeb spaces with non-trivial top homology groups and with good topological properties. This paper presents new families, especially, generalized situations of some known situations.

%Morse functions always exist densely on closed manifolds and singular points, appearing discretely, tell us information on homology groups and some information on homotopy of the manifolds. As a specific case, Reeb's theorem on Morse functions on closed manifolds states that a homotopy sphere is characterized as a manifold admitting just $2$ singular points except $4$-dimensional homotopy spheres:  a $4$-dimensional standard sphere is characterized in this way.

%  For example, we investigate algebraic topological restrictions on {\it Reeb spaces} of the fold maps. %They are essential tools in studying manifolds by using generic maps, and the source manfolds. We also %show flexibility of homology groups of the Reeb spaces and the source manifolds, for example.    

% Abstract text, usually no more than 200 words.
% Avoid bibliographic references (\cite) and complicated mathematics.
% Please do not use custom macros here, as this abstract has to 
% be able to stand alone.  You may use standard tex/latex/AMS macros.
\end{abstract}

% Leave these items like this, and the journal will fill them in.
% \received{Month Day, Year}   % receive date (for example: October 11, 1999)
% \revised{Month Day, Year}    % date of revision; omit, if no revision;
%                             % if multiple revisions, separate by commas
% \published{Month Day, Year}  % publish date\submitted{Bill Murray}      % Name of Journal's Editor, 
% who handled Article 
% \volumeyear{2014} % Volume Year
% \volumenumber{16} % Volume Number 
% \issuenumber{2}   % Issue Number
% \startpage{1}     % PageNumber of first page
% \articlenumber{1} % Sequence number of article within issue
% If copyright is retained by author, comment this out:
% \owner{International Press}

\maketitle
\section{Introduction.}
\label{sec:1}

For a continuous map $c:X \rightarrow Y$, we can define
the equivalence relation ${\sim}_c$ on $X$ so that $x_1 {\sim}_c x_2$ if and only if they are in a same connected component of a preimage. We can define the {\it Reeb space} of $c$ as follows.

\begin{Def}
In the situation above, the quotient space $W_c:=X/{\sim}_c$ is the Reeb space of $c$.
\end{Def}

$q_f:X \rightarrow W_c$ denotes the quotient map and $\bar{c}$ denotes the map uniquely obtained so that $c=\bar{c} \circ q_c$.

Reeb spaces are strong tools in (differential) topological theory of Morse functions, {\it fold} maps, which are their higher dimensional variants, and so on: they are in general polyhedra whose dimensions are same as those of the targets.
In suitable cases Reeb spaces inherit topological information such as homology groups, cohomology rings, and so on, of the manifolds. They play important roles in applied mathematics or applications of mathematical science such as data analysis and visualizations. \cite{sakuraisaekicarrwuyamamotoduketakahashi} is a related article.

Hereafter, $p \in X$ is said to be a {\it singular} point if for a differentiable map $c:X \rightarrow Y$ the rank of the differential ${dc}_p$ there is smaller than both the dimensions of the domain and the target. The {\it singular set} of a differentiable map $c$ is the set of all singular points, denoted by $S(c)$. The {\it singular value set} of $c$ means the image of the singular set. A {\it singular value} of $c$ is a point in the singular value set. The {\it regular value set} of $c$ means the complementary set of the singular value set. A {\it  regular value} of $c$ is a point in the regular value set.

For Morse functions and smooth functions which are not so wild, the Reeb spaces are graphs where the vertex sets are the sets of all (elements representing) connected components of preimages containing at least one {\it singular} point: see \cite{saeki4} for example. Reeb spaces are in such cases {\it Reeb graphs}.

We present simplest Morse functions and fold maps and their Reeb spaces. {\it Fold} maps are reviewed in the next section.

\begin{Ex}
For Morse functions on closed manifolds with exactly two singular points, if the dimensions of the manifolds are greater than $1$, then the Reeb spaces are graphs with exactly two vertices and one edge.

Canonical projections of unit spheres are simplest examples of fold maps and {\it special generic} maps: we review the class of special generic maps as an important subclass in the next section. 
For such a map, the singular set is an equator. The images are unit discs.
   
These maps are higher dimensional versions of Morse functions on unit spheres obtained by considering the natural heights: these functions are simplest ones in the functions before.

If the dimensions of the manifolds are greater than the dimensions of the targets, then the Reeb spaces are unit discs of the targets.

For a copy of the 2-dimensional torus embedded naturally in the $3$-dimensional Euclidean space, the Morse function obtained by the natural height has exactly four singular points and the Reeb graph is as follows.

\begin{enumerate}
\item The vertex set consists of exactly four vertices. For exactly two vertices, the degrees are $1$ and for the remaining vertices, the degrees are $3$.
\item The edge set consists of exactly four edges.
\item The $j$-th integral homology group of the Reeb space is isomorphic to $\mathbb{Z}$ for $j=0,1$. 
\end{enumerate} 
\end{Ex}

We present two problems on Reeb spaces.

\begin{Prob}

For a given graph, can we construct a good smooth function of a given suitable class?

\end{Prob}

\cite{sharko} is a pioneering study. 
\cite{masumotosaeki}, \cite{michalak}, \cite{saeki4}, and so on, are important works and there exist other closely related works. \cite{kitazawa8}, \cite{kitazawa10} and \cite{kitazawa11} are related works by the author. 

The following problem is related to the present study more.

\begin{Prob}

Construct good Morse functions, fold maps, and more general good smooth maps and investigate the topologies of their Reeb spaces?

\end{Prob}

Related to this, \cite{kobayashisaeki} concentrates on Reeb spaces of so-called {\it stable} maps into the plane on closed smooth manifolds whose dimensions are greater than $2$. It is shown that for $2$-dimensional polyhedra whose homology groups are isomorphic to that of a disc under the situation that the coefficient is isomorphic to $\mathbb{Z}/2\mathbb{Z}$ and having suitable  local properties can be Reeb spaces of stable maps of suitable classes. {\it Stable} maps are explained in the next section as ones forming an important subclass of smooth maps.

The author have studied explicit topologies of Reeb spaces of explicit fold maps via construction of the maps first in \cite{kitazawa} and later in \cite{kitazawa7}, \cite{kitazawa9}, \cite{kitazawa12}, \cite{kitazawa14}, \cite{kitazawa15}, and so on: for closely related studies on manifolds admitting such maps see \cite{kitazawa13} and \cite{kitazawa16} for example. Such studies are important and difficult and are also different from systematic theory of knowing existence of such maps in \cite{eliashberg}, \cite{eliashberg2}, and so on. Their homology groups and cohomology rings are various. Moreover, for example, top-homology groups are not trivial and free. As another important remark, these maps and Reeb spaces are obtained via very explicit and simple operations changing maps and manifolds locally starting from very fundamental fold maps such as canonical projections of unit spheres. 

The present paper concentrates on the following problem.

\begin{MainProb}

Construct explicit {\rm (}families of{\rm )} good Morse functions, fold maps, and more general smooth maps satisfying suitable differential topological properties and presenting Reeb spaces which do not collapse to lower dimensional polyhedra. Investigate their topologies.

\end{MainProb}

Main Theorems are as follows, leaving expositions on several undefined notions and notation later.

The following is of a new type.

\begin{MainThm}[Theorem \ref{thm:3}.]
Let $l \geq 0$ and $m>n \geq 1$ be integers and $\{X_j\}_{j=1}^l$ be a family of finitely many $n$-dimensional closed, connected and stably parallelizable manifolds Let $\{(F_{j,1},F_{j,2})\}_{j=1}^l$ be a family of pairs of {\rm (}$m-n${\rm )}-dimensional closed, connected and smooth manifolds for each of which a disjoint union of the two manifolds bounds a compact, connected and smooth manifold $F_j$ obtained by attaching {\rm handles} containing at most $1$ {\rm $1$-handle} to $(F_{j,1} \sqcup F_{j,2}) \times (\{1\} \subset [0,1])$. Let $X$ be an $n$-dimensional connected and compact manifold we can smoothly immerse into ${\mathbb{R}}^n$. Then we have a simple fold map $f$ on an $m$-dimensional closed and connected manifold $M$ into ${\mathbb{R}}^n$ satisfying the following two.
\begin{enumerate}
\item The Reeb space $W_f$ is a {\rm branched manifold} of {\rm the class A} simple homotopy equivalent to a space obtained by a finite iteration of taking a bouquet starting from $l+1$ polyhedra where the $j$-th polyhedron collapses to $X_j$ for $1 \leq j \leq l$ and is $X$ for $j=l+1$.
\item There exist connected components of the preimage of a regular value diffeomorphic to $F_{j,1}$ and $F_{j,2}$ for each $j$.
\end{enumerate}
\end{MainThm}

The following two generalize situations of ones of existing results, presented as Theorems \ref{thm:1} and \ref{thm:2}, abstractly.
\begin{MainThm}[Theorem \ref{thm:4}.]
Let $l \geq 0$ and $m>n \geq 1$ be integers.
Let $\{(F_{j,1},F_{j,2})\}_{j=1}^l$ be a family of pairs of {\rm (}$m-n${\rm )}-dimensional closed, connected and smooth manifolds for each of which a disjoint union of the two manifolds bounds a compact, connected and smooth manifold $F_j$ obtained by attaching handles containing at most $1$ $1$-handle to $(F_{j,1} \sqcup F_{j,2}) \times (\{0\} \subset [0,1])$. Let $X$ be an $n$-dimensional connected and compact manifold we can smoothly immerse into ${\mathbb{R}}^n$. Let $\{Y_j \subset X\}_{j=1}^l$ be a family of mutually disjoint $n$-dimensional compact, connected and smooth submanifolds of $X$. Then we have a simple fold map $f$ on an $m$-dimensional closed and connected manifold $M$ into ${\mathbb{R}}^n$ such that the Reeb space $W_f$ is a branched manifold of the class A obtained by attaching $l$ manifolds each of which is diffeomorphic to a double $DY_j$ of $Y_j$ identifying $Y_j$ in $X$ and $Y_j \subset DY_j$ canonically and that there exist connected components of the preimage of a regular value diffeomorphic to $F_{j,1}$ and $F_{j,2}$ for each $j$.
\end{MainThm}
\begin{MainThm}[Theorem \ref{thm:5}.]
In Theorem \ref{thm:4}, assume that $l>0$, that $X$ is obtained by attaching $h_p$ $p$-handles for $1 \leq p \leq n-1$ starting from $l$ $0$-handles and that the following four hold.
\begin{enumerate}
\item $H_1(X;\mathbb{Z}) \cong {\mathbb{Z}}^{h_1-(l-1)}$.
\item $H_p(X;\mathbb{Z}) \cong {\mathbb{Z}}^{h_p}$ for $2 \leq p \leq n-1$.
\item $F_j$ is obtained by attaching $h_{j,p}$ $p$-handles for $1 \leq p \leq n-1$ starting from a $0$-handle.
\item $X-{\sqcup}_{j=1}^l {\rm Int} Y_j$ is obtained by attaching $h_p-{\Sigma}_{j=1}^{l} h_{j,p}$ $p$-handles to the product of a manifold diffeomorphic to the boundary and $\{0\} \subset [0,1]$ for all $1 \leq p \leq n-1$.
\end{enumerate}
Then, we have the following facts on homology groups and cohomology groups and rings.
\begin{enumerate}
\item $H_p(W_f;\mathbb{Z})$ is isomorphic to the direct sum of $H_p(X;\mathbb{Z}) \oplus {\mathbb{Z}}^{{\Sigma}_{j=1}^l h_{j,n-p}}$ for $1 \leq p \leq n-1$ and ${\mathbb{Z}}^l$ for $p=n$.
\item $H^p(W_f;\mathbb{Z})$ is isomorphic to $A_p \oplus {\oplus}_{j=1}^l (H^{p}(Y_j;\mathbb{Z})) \oplus {\oplus}_{j=1}^l (H^{p}(DY_j-{\rm Int} Y_j;\mathbb{Z}))$ where $A_p$ is isomorphic to ${\mathbb{Z}}^{h_1-{\Sigma}_{j=1}^{l} h_{j,1}-(l-1)}$ for $p=1$ and ${\mathbb{Z}}^{h_p-{\Sigma}_{j=1}^{l} h_{j,p}}$ for $2 \leq p \leq n-1$. $H^n(W_f;\mathbb{Z})$ is isomorphic to ${\mathbb{Z}}^l$.
\item $H^{\ast}(X;\mathbb{Z})$ is regarded as a subalgebra of $H^{\ast}(W_f;\mathbb{Z})$ and the cohomology group is identified with $A_p \oplus {\oplus}_{j=1}^l (H^{p}(Y_j;\mathbb{Z})) \oplus \{0\} \subset A_p \oplus {\oplus}_{j=1}^l (H^{p}(Y_j;\mathbb{Z})) \oplus {\oplus}_{j=1}^l (H^{p}(DY_j-{\rm Int} Y_j;\mathbb{Z}))$ before for each degree $1 \leq p \leq n-1$. $H^{\ast}({\sqcup} DY_j;\mathbb{Z})$ is regarded as a subalgebra of $H^{\ast}(W_f;\mathbb{Z})$ and the cohomology group is identified with $\{0\} \oplus {\oplus}_{j=1}^l (H^{p}(Y_j;\mathbb{Z})) \oplus {\oplus}_{j=1}^l (H^{p}(DY_j-{\rm Int} Y_j;\mathbb{Z})) \subset A_p \oplus {\oplus}_{j=1}^l (H^{p}(Y_j;\mathbb{Z})) \oplus {\oplus}_{j=1}^l (H^{p}(DY_j-{\rm Int} Y_j;\mathbb{Z}))$ for each degree $1\leq p \leq n-1$ and ${\oplus}_{j=1}^l (H^{n}(DY_j;\mathbb{Z})) \cong {\mathbb{Z}}^l$ for degree $p=n$.
\item For $1 \leq p_1,p_2 \leq n-1$, the product of an element of degree $p_1$, identified with an element of $A_{p_1} \oplus \{0\} \subset A_{p_1} \oplus {\oplus}_{j=1}^l (H^{p_1}(Y_j;\mathbb{Z})) \oplus {\oplus}_{j=1}^l (H^{p_1}(DY_j-{\rm Int} Y_j;\mathbb{Z}))$ and an element of degree $p_2$, identified with an element of $\{0\} \oplus {\oplus}_{j=1}^l (H^{p_2}(DY_j-{\rm Int} Y_j;\mathbb{Z})) \subset A_{p_2} \oplus {\oplus}_{j=1}^l (H^{p_2}(Y_j;\mathbb{Z})) \oplus {\oplus}_{j=1}^l (H^{p_2}(DY_j-{\rm Int} Y_j;\mathbb{Z}))$, vanishes in $H^{\ast}(W_f;\mathbb{Z})$.
\end{enumerate}
\end{MainThm}

In the next section, we review the definition and fundamental properties of {\it fold} maps and their Reeb spaces. We also review {\it special generic} maps as an important subclass and their Reeb spaces, which are compact and smoothly immersed manifolds in the Euclidean spaces. The class of {\it simple} fold maps extends the class in a natural way. The Reeb spaces of such maps are so-called {\it branched} manifolds. We also refer to {\it stable} maps as an important subclass of smooth maps as a kind of appendices.
The third section is devoted to Main Theorems with supporting existing studies.

Hereafter, manifolds, maps between manifolds, (boundary) connected sums, and so on, are considered in the smooth category (or of the class $C^{\infty}$), unless otherwise stated. 
However, we also discuss in the PL, or as an equivalent category, the piecewise smooth category, For example, we consider PL bundles and smooth bundles as bundles whose fibers are smooth or PL manifolds or general polyhedra.

\section{(Stable) fold maps and their Reeb spaces.}
\begin{Def}
\label{def:2}
Let $m \geq n \geq 1$ be integers. Let $M$ be an $m$-dimensional closed manifold and $N$ be an $n$-dimensional manifold with no boundary. A smooth map $f:M \rightarrow N$ is said to be a {\it fold} map if at each singular point $p$ $f$ is represented as
$$(x_1,\cdots,x_m) \rightarrow (x_1,\cdots,x_{n-1},{\Sigma}_{j=n}^{m-i(p)} {x_j}^2-{\Sigma}_{m-i(p)+1}^{m} {x_j}^2)$$
 via suitable coordinates and a suitable integer $0 \leq i(p) \leq \frac{m-n+1}{2}$.
\end{Def}
For systematic studies and advanced expositions on fold maps, see \cite{golubitskyguillemin} and \cite{saeki} for example.
Morse functions are for specific cases or $n=1$ with $N=\mathbb{R}$. \cite{milnor} and \cite{milnor2} present fundamental information on Morse functions and as a related notion and technique, {\it $j$-handle} and an handle attachment in \cite{milnor2} and so on.

\begin{Prop}
\label{prop:1}
In Definition \ref{def:2}, for a fold map $f$, $i(p)$ is unique{\rm :} we call the integer the {\rm index} of $p$.
The restrictions to the singular set $S(f)$ and the set of all singular points of a fixed index, which are proven to be {\rm (}$n-1${\rm )}-dimensional closed and smooth submanifolds with no boundaries, are immersions.
\end{Prop}
\begin{Def}
\label{def:3}
A fold map is said to be {\it special generic} if the indices of singular points are always $0$.
\end{Def}

A smooth manifold is well-known to be regarded as an object in the PL category or a polyhedron canonically. Hereafter, for a smooth manifold, a polyhedron means this. 
For the following, refer to \cite{saeki2}, \cite{saeki3} and \cite{shiota} (or fact \ref{fact:1}).

\begin{Prop}
\label{prop:2}
Let $m>n \geq 1$ be integers. Let $M$ be an $m$-dimensional closed manifold and $N$ be an $n$-dimensional manifold with no boundary. For a fold map $f$, the Reeb space is an $n$-dimensional polyhedron uniquely induced from the manifold $N$ and $q_f$ and $\bar{f}$ are piecewise smooth maps.
The Reeb space of a special generic map $f$ is a compact manifold smoothly immersed into ${\mathbb{R}}^n$ via $\bar{f}$. 
\end{Prop}
\begin{Def}
\label{def:4}
A fold map $f$ is said to be {\it simple} if the restriction of $q_f$ to $S(f)$ is injective,
\end{Def}
See \cite{saeki} and \cite{sakuma} for example for simple fold maps.
In the following proposition, the Reeb space is not a manifold in general: it may be a so-called {\it branched manifold}. We  define a {\it branched manifold} in Definition \ref{def:8} as a polyhedron locally PL homeomorphic to a Reeb space in Propositions \ref{prop:5} and \ref{prop:6}. $2$-dimensional polyhedra of this kind are studied in various articles. \cite{turaev} is a pioneering work on so-called $2$-dimensional {\it simple polyhedra} or so-called {\it shadows}. \cite{ishikawakoda}, \cite{martelli}, \cite{naoe}, and so on, are closely related to this. 
These polyhedra form a class wider than that of the $2$-dimensional branched manifolds here. \cite{munozozawa} and \cite{ozawa} are on {\it multibranched surfaces} and regarded as slightly generalized versions of these 2-dimensional branched manifolds.
\begin{Prop}
\label{prop:3}
A special generic map is simple. Let $m>n \geq 1$ be integers. Let $M$ be an $m$-dimensional closed manifold and $N$ be an $n$-dimensional manifold with no boundary. For a simple fold map $f:M \rightarrow N$, $W_f$ is a branched manifold.
\end{Prop}

Last we add an exposition on so-called ({\it topologically}) {\it stable} maps as a kind of appendices. See \cite{golubitskyguillemin} for example for systematic expositions.
\begin{Def}
\label{def:5}
A smooth map $f$ between manifolds $M$ and $N$ with no boundaries is said to be a {\it stable} map if there exists an open neighborhood in the space of all smooth maps from $M$
 into $N$ endowed with the so-called {\it Whitney $C^{\infty}$ topology}, and for any map $f^{\prime}$ there, there exists a pair $({\Phi}_{f^{\prime}},{\phi}_{f^{\prime}})$ of diffeomorphisms satisfying ${\phi}_{f^{\prime}} \circ f^{\prime}=f \circ {\Phi}_{f^{\prime}}$. If we replace the diffeomorphisms by homeomorphisms, then $f$ is said to be {\it topologically stable}.
\end{Def}
The following fact is a very general fact. We omit the definition of a {\it Thom} map and precise expositions.
\begin{Fact}[\cite{shiota}.]
\label{fact:1}
For a fold map, {\rm (}topologically{\rm )} stable map, and more generally, a so-called {\rm Thom} map, the Reeb space is a polyhedron.
\end{Fact}
For smooth functions of suitable classes, stable maps into the plane on closed manifolds whose dimensions are greater than $2$, this has been explicitly shown in \cite{kobayashisaeki} and so on.
\begin{Prop}
\label{prop:4}
A fold map $f$ in Definition \ref{def:2} is stable if and only if for each value $p \in f(S(f))$ the following conditions hold.
\begin{enumerate}
\item The preimage $f^{-1}(p)$ consists of exactly $l \geq 1$ points and is denoted by $\{p_j\}_{j=1}^l$.
\item The dimension of the intersection of all images of the differential ${d_f}_{p_j}$ is $n-l$.
\end{enumerate}
\end{Prop}
\begin{Def}
\label{def:6}
An immersion of a smooth manifold with no boundary into another smooth manifold with no boundary is said to have {\it normal crossings only}, if this satisfies the two conditions the restriction $f {\mid}_{S(f)}$ in Proposition \ref{prop:4} satisfies (where we change these two in a suitable way to apply here). 
\end{Def}
\begin{Ex}
A Morse function such that the preimage of a singular value always has exactly one singular point is stable. On the other hand, a Morse function is stable if and only if this condition holds. It is a well-known fundamental fact that such Morse functions exist densely on any (closed) manifold. 
A fold map such that the restriction to the singular set is an embedding is stable.
\end{Ex}
In fact, a fold map can be deformed into a stable one by a slight perturbation or we can find a stable fold map close to the original map under the Whitney $C^{\infty}$ topology.
Moreover, we concentrate on simple fold maps in the present study and for such maps, this notion does not affect on our arguments essentially.
The following is an important proposition in knowing homology groups and cohomology rings of the manifolds from Reeb spaces in specific cases.

Hereafter, a {\it standard sphere} is a smooth manifold which is diffeomorphic to a unit sphere $S^k$: $S^k$ denotes the $k$-dimensional unit sphere. For the $k$-dimensional Euclidean space ${\mathbb{R}}^k$ and $p \in {\mathbb{R}}^k$, $||p||$ denotes the distance between $p$ and the origin $0$ where the Euclidean space is endowed with the standard Euclidean metric.

A {\it linear} bundle is a smooth bundle whose fiber is diffeomorphic to a unit disc $D^k:=\{x \mid ||x|| \leq 1\} \subset {\mathbb{R}}^k$ or a unit sphere $S^k:=\{x \mid ||x||=1\} \subset {\mathbb{R}}^{k+1}$ and whose structure group acts on the fiber linearly in a canonical way.

\begin{Prop}
\label{prop:5}
Let $m>n \geq 1$ be integers. Let $M$ be an $m$-dimensional closed manifold and $N$ be an $n$-dimensional manifold with no boundary. Let $f$ be a simple fold map. Assume also the following three.
\begin{enumerate}
\item Indices of singular points of $f$ are always $0$ or $1$.
\item Preimages of regular values are disjoint unions of copies of standard spheres.
\item $M$ is orientable if $m-n>1$.
\end{enumerate}
Then we have the following five.
\begin{enumerate}
\item For the image $q_f(C)$ of each connected component $C$ of the singular set consisting of singular points of index $0$, there exists a regular neighborhood PL homeomorphic to $q_f(C) \times [0,1]$ where $q_f(C)$ is identified canonically with $q_f(C) \times \{0\}$. Furthermore, the composition of the restriction of $q_f$ to the preimage of the regular neighborhood with the canonical projection to $q_f(C)$ gives a linear bundle whose fiber is diffeomorphic to $D^{m-n+1}$.
\item For the image $q_f(C)$ of each connected component $C$ of the singular set consisting of singular points of index $1$, there exists a regular neighborhood PL homeomorphic to the total space of a PL bundle over $q_f(C)$ whose fiber is PL homeomorphic to $K:=\{(r\cos t,r \sin t) \mid 0 \leq r \leq 1, t=\frac{2}{3} a \pi, a=0,1,2.\} \subset {\mathbb{R}}^2$. Furthermore the following three must hold.
\begin{enumerate}
\item The structure group is isomorphic to the symmetric group $S_2$ of degree $2$ preserving the distance $||p||$ between each point $p \in K$ and the origin $0$ and fixing $\{(r,0) \mid 0 \leq r \leq1.\}$.
\item $q_f(C)$ is identified canonically with $q_f(C) \times \{0\}$ in the bundle.
\item Furthermore, the composition of the restriction of $q_f$ to the preimage of the regular neighborhood with the canonical projection to $q_f(C)$ gives a smooth bundle whose fiber is diffeomorphic to a manifold obtained by removing three disjointly and smoothly embedded copies of $D^{m-n+1}$ in a copy of $S^{m-n+1}$. 
\end{enumerate}
\item $q_f$ induces the isomorphisms ${q_f}_{\ast}:H_j(M;A) \rightarrow H_j(W_f;A)$, ${q_f}^{\ast}:H^j(W_f;A) \rightarrow H^j(M;A)$ and ${q_f}_{\ast}:{\pi}_j(M) \rightarrow {\pi}_j(W_f)$ for $0 \leq j<m-n$. If $f$ is special generic map, then these isomorphisms work for $j=m-n$.
\item ${q_f}^{\ast}:H^{\ast}(W_f;A) \rightarrow H^{\ast}(M;A)$ preserves cup products for any pair such that the sum of the degrees is smaller than $m-n$ and we can replace ''smaller than'' by ''smaller than or equal to'' for a special generic map $f$.
\item Cup products of $M$ for any pair such that the degree of each class is smaller than $m-n$ and that the sum of the degrees is greater than $n$ always vanish. For a special generic map $f$, we can replace ''smaller than'' by ''smaller than or equal to'' and ''greater than'' by ''greater than or equal to''.
\end{enumerate}
\end{Prop}
\begin{Def}
\label{def:7}
A stable fold map $f$ satisfying the assumption of Proposition \ref{prop:5} is said to be {\it standard-spherical}.
\end{Def}
This class, extending the class of special generic maps naturally, has been studied in various scenes. See \cite{kitazawa0.1}, \cite{kitazawa0.2}, \cite{kitazawa0.5}, \cite{kitazawa4}, \cite{kitazawa5}, \cite{kitazawa6}, \cite{kitazawa13}, \cite{kitazawa16}, \cite{kitazawa17}, \cite{kitazawa18}, \cite{saekisuzuoka} and so on.
We also present a general proposition as Proposition \ref{prop:6}, regarded as a proposition explaining Proposition \ref{prop:3} more precisely.

\begin{Prop}
\label{prop:6}
Let $m>n \geq 1$ be integers. Let $M$ be an $m$-dimensional closed manifold and $N$ be an $n$-dimensional manifold with no boundary. Let $f$ be a simple fold map. 
Then we have the following three.
\begin{enumerate}
\item For the image $q_f(C)$ of each connected component $C$ of the singular set consisting of singular points of index $0$, there exists a small regular neighborhood PL homeomorphic to $q_f(C) \times [0,1]$ where $q_f(C)$ is identified canonically with $q_f(C) \times \{0\}$. Furthermore, the composition of the restriction of $q_f$ to the preimage of the regular neighborhood with the canonical projection to $q_f(C)$ gives a linear bundle whose fiber is diffeomorphic to $D^{m-n+1}$.
\item For the image $q_f(C)$ of each connected component $C$ of the singular set consisting of singular points of index $j>0$, there exists a small regular neighborhood PL homeomorphic to the total space of a PL bundle over $q_f(C)$
such that fibers satisfy either of the following two. Furthermore, the composition of the restriction of $q_f$ to the preimage of the regular neighborhood with the canonical projection to $q_f(C)$ gives a smooth bundle.
\begin{enumerate}
\item The fiber is PL homeomorphic to $K:=\{(r\cos t,r \sin t) \mid 0 \leq r \leq 1, t=\frac{2}{3} a \pi, a=0,1,2.\} \subset {\mathbb{R}}^2$. Furthermore the following two must hold.
\begin{enumerate}
\item The structure group is isomorphic to the symmetric group $S_2$ of degree $2$ preserving the distance $||p||$ between each point $p \in K$ and the origin $0$ and fixing $\{(r,0) \mid 0 \leq r \leq1.\}$.
\item $q_f(C)$ is identified canonically with $q_f(C) \times \{0\}$ in the bundle.
\end{enumerate}
\item The fiber is PL homeomorphic to $I:=[-1,1] \subset \mathbb{R}$. Furthermore the following two must hold.
\begin{enumerate}
\item The structure group is trivial. 
\item $q_f(C)$ is identified canonically with $q_f(C) \times \{0\}$ in the bundle.
\end{enumerate}
\end{enumerate}
\item The composition of the restriction of $q_f$ to the preimage of the regular neighborhood with the canonical projection to $q_f(C)$ gives a smooth bundle. 
\end{enumerate}
\end{Prop}

\section{Main Theorems and existing supporting studies.}
Let $n$ be a positive integer. Let $\{G_j\}_{j=0}^n$ be a sequence of finitely generated commutative groups of length $n$ such that $G_0$ is trivial, that $G_{n-1}$ is free, that $G_n$ is free and non-trivial, and that a suitable additional condition (A) holds.  

Motivated by the importance and the difficulty of construction of explicit fold maps and obtaining information of the resulting manifolds, the author has obtained several results of the following types via {\it bubbling operations}, which are surgery operations to manifolds and maps. The property (B) denotes a suitable property on the structure of the (resulting) fold maps.
\begin{Thm}
\label{thm:1}
Let $f_0:M_0 \rightarrow {\mathbb{R}}^n$ be a stable fold map on an $m$-dimensional closed and connected manifold $M_0$ into ${\mathbb{R}}^n$ satisfying $m \geq n \geq 1$.

Then by a finite iteration of bubbling operations starting from $f_0$, we have a stable fold map $f$ satisfying the property {\rm (}B{\rm )} and $H_j(W_f;\mathbb{Z}) \cong H_j(W_{f_0};\mathbb{Z}) \oplus G_j$ for $0 \leq j \leq n$.

Especially, if $f_0$ is standard-spherical, then we can construct a standard-spherical map $f$ in this way.  

\end{Thm}

\begin{Thm}
\label{thm:2}
Let $f_0:M_0 \rightarrow {\mathbb{R}}^n$ be a stable fold map on an $m$-dimensional closed and connected manifold $M_0$ into ${\mathbb{R}}^n$ satisfying $m \geq n \geq 1$.

Then by finite iterations of bubbling operations, we have a family $\{f_{\lambda}:M_{\lambda} \rightarrow {\mathbb{R}}^n\}$ of infinitely many stable fold maps satisfying the property {\rm (}B{\rm )} and $H_j(W_{f_{\lambda}};\mathbb{Z}) \cong H_j(W_{f_0};\mathbb{Z}) \oplus G_j$ for $0 \leq j \leq n$ and for distinct ${\lambda}_1$ and ${\lambda}_2$, the integral cohomology rings of the resulting Reeb spaces are mutually non-isomorphic.

As in Theorem \ref{thm:1}, if the given map $f_0$ is standard-spherical, then we can construct these fold maps as standard-spherical maps.

\end{Thm}

We review a {\it bubbling operation} shortly. 

First, in a connected component of the regular value set, we choose a suitable closed, and connected smooth submanifold with no boundary or a polyhedron represented as one obtained by a finite iteration of taking a bouquet of two such submanifolds starting from finitely many such submanifolds. We remove the interior of a connected component of the preimage of a small regular neighborhood of this. We attach a new smooth map so that the singular value set is the disjoint union of the original singular value set and a connected submanifold diffeomorphic and parallel to the boundary made after the removal.
It is motivated by several operations for construction of new stable (fold) maps in \cite{kobayashi}, \cite{kobayashi2} and \cite{kobayashisaeki}. Especially, a {\it bubbling surgery}, first defined and systematically presented in \cite{kobayashi2}, has motivated the author to study in this way. A bubbling surgery is a case where the submanifold is a point.

This is extended for cases where connected submanifolds or submanifolds in subpolyhedra are images of manifolds via immersions having normal-crossings only and cases where the submanifolds are standard spheres are studied in \cite{kitazawa14} and \cite{kitazawa15}. In the present paper, we do not concentrate on such cases.

For studies related to this, see \cite{kitazawa}, \cite{kitazawa7}, \cite{kitazawa9}, \cite{kitazawa12}, \cite{kitazawa14}, \cite{kitazawa15}, and so on. For ones concentrating more on manifolds admitting the maps, see \cite{kitazawa13}, \cite{kitazawa16}, and so on.  
\cite{kitazawa0.1}--\cite{kitazawa0.5} are on {\it round} fold maps: they are, shortly, stable fold maps such that the singular values are embedded concentric spheres and canonical projections of unit spheres are simplest round fold maps. This class forms an important subclass of fold maps, introduced first by the author. In \cite{kobayashi2} and so on, Kobayashi has studied this class without defining the class. This class has also motivated the author to introduce the bubbling operations, extending construction of some round fold maps.

We present results or regarded as ones on global topologies of Reeb spaces of stable fold maps of suitable classes. 

\cite{masumotosaeki} and so on say, that a graph with no loops is regarded as the Reeb space of a smooth function of a suitable class on a closed surface. Posing additional conditions on the graphs, the graphs can be regarded as the Reeb graphs of Morse functions satisfying suitable conditions on closed surfaces and manifolds. Furthermore, if graphs may have loops, then by replacing the targets with circles, similar facts hold.

We explain higher dimensional cases. There exist several known related facts for $2$-dimensional cases. We introduce several notions. Respecting Proposition \ref{prop:6}, we can define a {\it branched} manifold of the {\it class A {\rm (}B{\rm )}}. It may be a notion to be defined in a more general form. However, in the present paper, we define the notion as follows.

\begin{Def}
\label{def:8}
A $k$-dimensional compact polyhedron $X$ satisfying the following two is said to be a {\it branched} manifold of {\it the class A}.
\begin{enumerate}
\item There exists a family $\{Y_j\}$ of finitely many compact and connected PL submanifolds with no boundaries.
\item $X-{\sqcup}_j Y_j$ is an $k$-dimensional manifold with no boundary.
\item For each $Y_j$, either of the following two holds. 
\begin{enumerate}
\item There exists a small regular neighborhood PL homeomorphic to $Y_j \times [0,1]$ where $Y_j$ is identified canonically with $Y_j \times \{0\}$.
\item There exists a small regular neighborhood PL homeomorphic to the total space of a PL bundle over $Y_j$
such that the fiber is PL homeomorphic to $K:=\{(r\cos t,r \sin t) \mid 0 \leq r \leq 1, t=\frac{2}{3} a \pi, a=0,1,2.\} \subset {\mathbb{R}}^2$. Furthermore the following two must hold.
\begin{enumerate}
\item The structure group is isomorphic to the symmetric group $S_2$ of degree $2$ preserving the distance $||p||$ between each point $p \in K$ and the origin $0$ and fixing $\{(r,0) \mid 0 \leq r \leq1.\}$.
\item $Y_j$ is identified canonically with $Y_j \times \{0\}$ in the bundle.
\end{enumerate}
\end{enumerate}
If we replace the phrase ''the symmetric group $S_2$ of degree $2$ preserving the distance $||p||$ of each point $p \in K$ and the origin $0$ and fixing $\{(r,0) \mid 0 \leq r \leq1.\}$''
 by ''the symmetric group $S_3$ of degree $3$ preserving the distance $||p||$ between each point $p \in K$ and the origin $0$'', then $X$ is said to be a {\it branched} manifold of {\it the class B}.
\end{enumerate}

\end{Def}

A ($2$-dimensional) {\it simple polyhedron} with no {\it vertices} is a branched manifold of the class B.  The following is essentially equivalent to Proposition 5.2 in \cite{kobayashisaeki}.

\begin{Fact}[Essentially shown in \cite{kobayashisaeki}.]
\label{fact:2}
Let $m>2$ be an integer.
If a $2$-dimensional branched manifold $X$ of the class $A$ satisfies $H_j(X;\mathbb{Z}/2\mathbb{Z}) \cong H_j(D^2;\mathbb{Z}/2\mathbb{Z})$ for any $j$, then it is PL homeomorphic to the Reeb space of a standard-spherical map on an $m$-dimensional closed and connected manifold into the plane. 
\end{Fact}
The following is closely related and to some extent equivalent to Fact \ref{fact:2} and a result proven by a different argument.

\begin{Fact}[\cite{naoe}]
\label{fact:3}
If a $2$-dimensional branched manifold $X$ of the class $B$ is simply-connected and satisfies $H_j(X;\mathbb{Z}) \cong H_j(D^2;\mathbb{Z})$ for any $j$, then it collapses to a point.
\end{Fact}
More precisely, he uses graphs with some labels which are fundamental and strong tools in representing the polyhedra, introduced in \cite{martelli}. 
He also shows that this is of the class A in proving this. 

Note also that Naoe shows this to prove important results in low-dimensional differential topology there: he has shown that $3$-dimensional closed manifolds {\it represented by such shadows} or admitting simple stable fold maps whose Reeb spaces are isomorphic to such polyhedra are $3$-dimensional standard spheres and that $4$-dimensional compact manifolds {\it represented by such shadows} are diffeomorphic to the $4$-dimensional unit disc. For shadows and such manifolds represented by shadows, see \cite{turaev} and see also \cite{ishikawakoda}, for example. Note also that Martelli's studies such as \cite{martelli} are closely related to these studies. 

From this, we can see that such a polyhedron is PL homeomorphic to the Reeb space of a standard-spherical map on an $m$-dimensional closed and connected manifold into the plane where $m \geq 3$. We show a sketch of a proof of this fact. In Definition \ref{def:8}, we construct a smooth map on the preimage of the small regular neighborhood of $Y_j$ as a smooth family of a Morse function over $Y_j$ ($Y_j \times \{0\}$) satisfying either of the following two respecting the topology of the regular neighborhood of $Y_j$.
\begin{enumerate}
\item A product map of a Morse function on a copy of $D^{m-1}$ which is obtained by considering a natural height and represented as the form $(x_1,\cdots,x_{m-1}) \rightarrow {\Sigma}_{j=1}^{m-1} {x_j}^2$ for suitable coordinates and the identity map on $Y_j \times \{0\}$ for $Y_j$ such that the regular neighborhood is PL homeomorphic to $Y_j \times [0,1]$.
\item A smooth family of a Morse function on a manifold obtained by removing the interiors of $3$ disjointly and smoothly embedded copies of an ($m-1$)-dimensional unit disc from an ($m-1$)-dimensional standard sphere on $Y_j \times \{0\}$ satisfying the following two for $Y_j$ such that the regular neighborhood is PL homeomorphic to the total space of a bundle over $Y_j$ whose fiber is PL homeomorphic to $K$.
\begin{enumerate}
\item The Morse function has exactly one singular point in the interior, the image is a closed interval and the preimage of the minimal value and that of the maximal value are exactly one connected component of the boundary and exactly two connected components of the boundary, respectively.
\item The Morse function is invariant under a suitable smooth action by $\mathbb{Z}/2\mathbb{Z}$: the family is trivial if the bundle over $Y_j \times \{0\}$ whose fiber is PL homeomorphic to $K$ is product and not trivial if the bundle over $Y_j \times \{0\}$ whose fiber is PL homeomorphic to $K$ is not trivial.
\end{enumerate}
\end{enumerate}
Last, on the preimage of the complementary set of the disjoint union of the interiors of the small regular neighborhoods for all $Y_j$'s, we construct the projection of a trivial bundle over the complementary set whose fiber is an ($m-2$)-dimensional standard sphere. A kind of induction respecting collapsing for simplicial complexes of the branched manifolds presents a desired map.

Main Theorems give new explicit construction of (simple) stable fold maps and their Reeb spaces, which are branched manifolds of the class A. 

Theorem \ref{thm:3} is of a new type. Theorems \ref{thm:4} and \ref{thm:5} are regarded as results generalizing situations of ones of existing results 
presented as Theorems \ref{thm:1}--\ref{thm:2} abstractly.

Before presenting them, we explain fundamental notions.

\begin{Def}
For a module $A$ over a commutative ring $R$ having a unique identity element $1$ different from the zero element $0$
\begin{enumerate}
\item If $a \in A$ is not represented as $a=ra^{\prime}$ for any pair of an element $r \in R$ which is not a unit and $a^{\prime} \in A$, then $a$ is said to be a {\it unit element}. 
\item For an arbitrary unit element $a \in A$, we can uniquely define a homomorphism $a^{\ast}$ into $R$ satisfying the following two and it is said to be the {\it dual} of $a$.
\begin{enumerate}
\item $a^{\ast}(a)=1$.
\item $a^{\ast}(b)=0$ for any submodule $B$ giving an internal direct sum decomposition $<a> \oplus B$ of $A$ and any $b \in B$ where $<a>$ is the submodule generated by $\{a\}$ ($a$).
\end{enumerate}
\end{enumerate}
\end{Def}
In the present discussions, duals of homology classes (for which we can define the duals) are important. They are regarded as cohomology classes in canonical ways.  
\begin{Def}
For a smooth or PL manifold $X$, a homology class $c$ is said to be {\it represented} by a closed and smooth or PL submanifold $Y$ with no boundary and with an orientation if $c$ is realized as the value of the homomorphism induced by the inclusion at the so-called {\it fundamental class}: the {\it fundamental class} of an oriented closed and smooth or PL manifold is the generator of the top homology group respecting the orientation and note that we do not need the orientation if the coefficient is isomorphic to $\mathbb{Z}/2\mathbb{Z}$. 
\end{Def}
A {\it stably parallelizable} manifold is a smooth manifold such that the Whitney sum of the tangent bundle and a $1$-dimensional trivial real bundle over the manifold is a trivial real vector bundle.
\begin{Thm}
\label{thm:3}
Let $l \geq 0$ and $m>n \geq 1$ be integers and $\{X_j\}_{j=1}^l$ be a family of finitely many $n$-dimensional closed, connected and stably parallelizable manifolds Let $\{(F_{j,1},F_{j,2})\}_{j=1}^l$ be a family of pairs of {\rm (}$m-n${\rm )}-dimensional closed, connected and smooth manifolds for each of which a disjoint union of the two manifolds bounds a compact, connected and smooth manifold $F_j$ obtained by attaching handles containing at most $1$ $1$-handle to $(F_{j,1} \sqcup F_{j,2}) \times (\{1\} \subset [0,1])$. Let $X$ be an $n$-dimensional connected and compact manifold we can smoothly immerse into ${\mathbb{R}}^n$. Then we have a simple fold map $f$ on an $m$-dimensional closed and connected manifold $M$ into ${\mathbb{R}}^n$ satisfying the following two.
\begin{enumerate}
\item The Reeb space $W_f$ is a branched manifold of the class A simple homotopy equivalent to a space obtained by a finite iteration of taking a bouquet starting from $l+1$ polyhedra where the $j$-th polyhedron collapses to $X_j$ for $1 \leq j \leq l$ and is $X$ for $j=l+1$.
\item There exist connected components of the preimage of a regular value diffeomorphic to $F_{j,1}$ and $F_{j,2}$ for each $j$.
\end{enumerate} 
\end{Thm}
\begin{proof}
Each $X_i$ admits a stable fold map into ${\mathbb{R}}^n$ by virtue of the theory of \cite{eliashberg} and as a result it is special generic such that the singular set is a closed submanifold of dimension $n-1$ with no boundary whose normal bundle is trivial and whose complementary set is not connected and consists of exactly two connected components: see also \cite{yamamoto} for example as an explicit study on \cite{eliashberg}, considering the theory for cases of $3$-dimensional closed and orientable manifolds.
For the singular set of each special generic map, take a small closed tubular neighborhood $NS(X_i)$, regarded as a trivial linear bundle over the singular set whose fiber is diffeomorphic to $D^1$. We construct a Morse function on $F_i$ satisfying the following two.
\begin{enumerate}
\item The preimage of the minimal value is the boundary $F_{i,1} \sqcup F_{i,2}$ and contains no singular points.
\item At distinct singular points, the singular values are distinct.
\item The Reeb space is PL homeomorphic to $K$ in Definition \ref{def:8} and so on. 
\end{enumerate} 
We have a product map of the Morse function and the identity map of the singular value set of the special generic map before. The Reeb space is PL homeomorphic to the product of $K$ and the singular set of the special generic map.

We have trivial smooth bundles over two disjoint compact and connected submanifolds of $X_i-{\rm Int} NS(X_i)$ restoring the original manifold $X_i-{\rm Int} NS(X_i)$ by the disjoint union whose fibers are diffeomorphic to $F_{i,1}$ and $F_{i,2}$, respectively. Note that the $n$-dimensional compact manifolds are smoothly immersed into ${\mathbb{R}}^n$ respecting the special generic map on $X_i$. We can glue these maps to obtain a simple fold map such that the Reeb space collapses to $X_i$ and that there exist connected components of the preimage of a regular value diffeomorphic to $F_{i,1}$ and $F_{i,2}$.
This technique generalizes a method in the proof of Theorem 1 and so on in \cite{kitazawa0.4}.

To obtain a desired map, we take a {\it connected sum} of $l$ maps obtained in this way and a special generic map on an $m$-dimensional closed and connected manifold whose Reeb space is diffeomorphic to $X$, smoothly immersed into ${\mathbb{R}}^n$. Note that such a special generic map is constructed by virtue of a fundamental argument in \cite{saeki2} or we can construct such a map easily. 
We do not review a rigorous definition of a {\it connected sum} of two or more simple (stable) fold maps or general stable (fold) maps. However, we can understand the definition easily in a natural way. This notion is fundamental in various scenes. In \cite{saeki2}, connected sums of special generic maps are fundamental and important, for example.
\end{proof}
\begin{Thm}
\label{thm:4}
Let $l \geq 0$ and $m>n \geq 1$ be integers.
Let $\{(F_{j,1},F_{j,2})\}_{j=1}^l$ be a family of pairs of {\rm (}$m-n${\rm )}-dimensional closed, connected and smooth manifolds for each of which a disjoint union of the two manifolds bounds a compact, connected and smooth manifold $F_j$ obtained by attaching handles containing at most $1$ $1$-handle to $(F_{j,1} \sqcup F_{j,2}) \times (\{0\} \subset [0,1])$. Let $X$ be an $n$-dimensional connected and compact manifold we can smoothly immerse into ${\mathbb{R}}^n$. Let $\{Y_j \subset X\}_{j=1}^l$ be a family of mutually disjoint $n$-dimensional compact, connected and smooth submanifolds of $X$. Then we have a simple fold map $f$ on an $m$-dimensional closed and connected manifold $M$ into ${\mathbb{R}}^n$ such that the Reeb space $W_f$ is a branched manifold of the class A obtained by attaching $l$ manifolds each of which is diffeomorphic to a double $DY_j$ of $Y_j$ identifying $Y_j$ in $X$ and $Y_j \subset DY_j$ canonically and that there exist connected components of the preimage of a regular value diffeomorphic to $F_{j,1}$ and $F_{j,2}$ for each $j$.
\end{Thm} 
\begin{proof}
First we construct a special generic map whose Reeb space is diffeomorphic to $X$, smoothly immersed into ${\mathbb{R}}^n$, as in Theorem \ref{thm:3}. We can do this so that the restriction to the preimage of the complementary space of the interior of a small collar neighborhood of $\partial X$ in the Reeb space is the projection of a trivial bundle over the space such that the fiber is diffeomorphic to $S^{m-n}$ by virtue of the theory.
We take $Y_j$ and a small closed tubular neighborhood of $\partial Y_j$ in $X$. We remove the interior of the preimage of the closed tubular neighborhood for each $j$ and we attach a new map.
We construct this new map as a product map of a Morse function satisfying the following four and the identity map on a manifold diffeomorphic to $\partial Y_j$.
\begin{enumerate}
\item The preimage of the minimal value is the disjoint union of two connected components of the boundary and diffeomorphic to $F_{i,1} \sqcup F_{i,2}$ and contains no singular points.
\item The preimage of the maximal value is a connected component of the boundary and a standard sphere and contains no singular points.
\item At distinct singular points, the singular values are distinct.
\item The Reeb space is PL homeomorphic to $K$ in Definition \ref{def:8} and so on. 
\end{enumerate}  
We have a collar neighborhood of $\partial Y_j \subset Y_j$ in this situation naturally. Over the complementary space of the interior of the collar neighborhood, we construct a trivial smooth bundle whose fiber is diffeomorphic to $F_{i,1} \sqcup F_{i,2}$ instead.
By gluing these maps and composing the resulting map with the original immersion of $X$ into ${\mathbb{R}}^n$, we have a desired simple fold map.
\end{proof}
\begin{Rem}
\label{rem:1}
We can obtain infinitely many stably parallelizable closed manifolds even if we restrict the class of the manifolds. For example, we can find such manifolds in the class of $4$--$7$-dimensional closed and simply-connected manifolds, studied or systematically explained in \cite{barden}, \cite{gompgstipsicz}, \cite{kreck}, \cite{smale}, \cite{wall}, \cite{zhubr}, \cite{zhubr2}, and so on. 
Some well-known such explicit manifolds are given. In addition, for example, by considering smooth immersions or embeddings into suitable Euclidean spaces whose codimensions are greater than $2$, we can consider normal bundles and the subbundles obtained by restricting the fibers to the unit spheres. The total spaces are such closed manifolds.
Such facts imply that we can find various homology groups, cohomology rings, and so on, of the resulting Reeb spaces.
\end{Rem}
\begin{Ex}
Related to Remark \ref{rem:1}, every $4$-dimensional closed, connected and orientable manifold can be smoothly embedded into ${\mathbb{R}}^7$. We can have $6$-dimensional closed, connected and stably parallelizable manifolds in the presented way for these manifolds. We can apply Theorems \ref{thm:3} and \ref{thm:4}: for example by taking $X_j$ as such a $6$-dimensional manifold, $X$ as a unit disc, and $n=6$ in Theorem \ref{thm:3}. Some $4$-dimensional closed, connected and orientable manifolds present Reeb spaces such that the square of some cohomology class may not be divisible by $2$. Such cases have not been studied in existing related studies (by the author).   
\end{Ex}
\begin{Rem}
\label{rem:2}
Related to Main Theorems, \cite{hiratukasaeki}, \cite{hiratukasaeki2}, \cite{kitazawa2}, \cite{kitazawa3}, and so on, present results stating that existence of a connected component of the preimage of a regular value which is not null-cobordant or which is not null-cobordant in a wider sense different from original well-known senses makes the top homology group of the Reeb space non-trivial for some coefficient.
\end{Rem}
\begin{Thm}
\label{thm:5}
In Theorem \ref{thm:4}, assume that $l>0$, that $X$ is obtained by attaching $h_p$ $p$-handles for $1 \leq p \leq n-1$ starting from $l$ $0$-handles and that the following four hold.
\begin{enumerate}
\item $H_1(X;\mathbb{Z}) \cong {\mathbb{Z}}^{h_1-(l-1)}$.
\item $H_p(X;\mathbb{Z}) \cong {\mathbb{Z}}^{h_p}$ for $2 \leq p \leq n-1$.
\item $F_j$ is obtained by attaching $h_{j,p}$ $p$-handles for $1 \leq p \leq n-1$ starting from a $0$-handle.
\item $X-{\sqcup}_{j=1}^l {\rm Int} Y_j$ is obtained by attaching $h_p-{\Sigma}_{j=1}^{l} h_{j,p}$ $p$-handles to the product of a manifold diffeomorphic to the boundary and $\{0\} \subset [0,1]$ for all $1 \leq p \leq n-1$.
\end{enumerate}
Then, we have the following facts on homology groups and cohomology groups and rings.
\begin{enumerate}
\item $H_p(W_f;\mathbb{Z})$ is isomorphic to the direct sum of $H_p(X;\mathbb{Z}) \oplus {\mathbb{Z}}^{{\Sigma}_{j=1}^l h_{j,n-p}}$ for $1 \leq p \leq n-1$ and ${\mathbb{Z}}^l$ for $p=n$.
\item $H^p(W_f;\mathbb{Z})$ is isomorphic to $A_p \oplus {\oplus}_{j=1}^l (H^{p}(Y_j;\mathbb{Z})) \oplus {\oplus}_{j=1}^l (H^{p}(DY_j-{\rm Int} Y_j;\mathbb{Z}))$ where $A_p$ is isomorphic to ${\mathbb{Z}}^{h_1-{\Sigma}_{j=1}^{l} h_{j,1}-(l-1)}$ for $p=1$ and ${\mathbb{Z}}^{h_p-{\Sigma}_{j=1}^{l} h_{j,p}}$ for $2 \leq p \leq n-1$. $H^n(W_f;\mathbb{Z})$ is isomorphic to ${\mathbb{Z}}^l$.
\item $H^{\ast}(X;\mathbb{Z})$ is regarded as a subalgebra of $H^{\ast}(W_f;\mathbb{Z})$ and the cohomology group is identified with $A_p \oplus {\oplus}_{j=1}^l (H^{p}(Y_j;\mathbb{Z})) \oplus \{0\} \subset A_p \oplus {\oplus}_{j=1}^l (H^{p}(Y_j;\mathbb{Z})) \oplus {\oplus}_{j=1}^l (H^{p}(DY_j-{\rm Int} Y_j;\mathbb{Z}))$ before for each degree $1 \leq p \leq n-1$. $H^{\ast}({\sqcup} DY_j;\mathbb{Z})$ is regarded as a subalgebra of $H^{\ast}(W_f;\mathbb{Z})$ and the cohomology group is identified with $\{0\} \oplus {\oplus}_{j=1}^l (H^{p}(Y_j;\mathbb{Z})) \oplus {\oplus}_{j=1}^l (H^{p}(DY_j-{\rm Int} Y_j;\mathbb{Z})) \subset A_p \oplus {\oplus}_{j=1}^l (H^{p}(Y_j;\mathbb{Z})) \oplus {\oplus}_{j=1}^l (H^{p}(DY_j-{\rm Int} Y_j;\mathbb{Z}))$ for each degree $1\leq p \leq n-1$ and ${\oplus}_{j=1}^l (H^{n}(DY_j;\mathbb{Z})) \cong {\mathbb{Z}}^l$ for degree $p=n$.
\end{enumerate}
\end{Thm}
\begin{proof}
We consider a Mayer-Vietoris sequence $$\rightarrow H_p({\sqcup}_{j=1}^l Y_j;\mathbb{Z}) \rightarrow H_p(X;\mathbb{Z}) \oplus H_p({\sqcup}_{j=1}^l DY_j;\mathbb{Z}) \rightarrow H_p(W_f;\mathbb{Z}) \rightarrow$$ and the
 homomorphism from $H_p({\sqcup}_{j=1}^l Y_j;\mathbb{Z})$ into $H_p(X;\mathbb{Z}) \oplus H_p({\sqcup}_{j=1}^l DY_j;\mathbb{Z})$ is a monomorphism. Furthermore, each summand is induced by the natural inclusion and a monomorphism.
 Assumptions on handles and integral homology groups, which are free, imply that in the argument we cannot cancel pairs of handles except $l-1$ pairs of $0$-handles and $1$-handles.
 $DY_j$ is represented as a double of $Y_j$. $D_{Y_j}-{\rm Int} Y_j$ is, by considering the duals to original handles for $Y_j$, regarded as a manifold obtained by attaching $h_{j,p}$ ($n-p$)-handles to the product of a manifold diffeomorphic to the boundary and $\{0\} \subset [0,1]$ for $1 \leq p \leq n-1$. These arguments on the topological structures of the manifolds complete the proof of all facts. 
\end{proof}
\begin{Ex}
\cite{kitazawa0.1}, \cite{kitazawa0.2}, \cite{kitazawa0.5}, and so on, present explicit examples implicitly or explicitly for Theorem \ref{thm:5} where $X$ and $Y_j$ are $n$-dimensional standard discs. 
We can know that manifolds represented as connected sums of total spaces of smooth bundles over the $n$-dimensional standard sphere whose fibers are ($m-n$)-dimensional standard spheres admit such maps with $m>n \geq 1$. \cite{kitazawa}, \cite{kitazawa7}, and so on, present explicitly or implicitly more general examples. 
\end{Ex}
\section{Acknowledgment with additional remarks related to applications of our studies on geometry of manifolds to machine-learnings and related topics and data availability.}
The author is a member of and supported by JSPS KAKENHI Grant Number JP17H06128 "Innovative research of geometric topology and singularities of differentiable mappings" (Principal Investigator: Osamu Saeki). 

This is also closely related to a joint research project at Institute of Mathematics for Industry, Kyushu University (20200027), ''Geometric and constructive studies of higher dimensional manifolds and applications to higher dimensional data'', principal investigator of which is the author. The author would like to thank people supporting the research project. 
This is a project on applications of mathematical (, especially, geometric) theory on higher dimensional differentiable manifolds developed through the studies of the author to data analysis, visualizations, and so on. 
This is a kind of new project of applying the singularity theory of differentiable maps and (differential) topology to machine-learnings and related problems such as multi-optimization problems, genetic algorithms, evolutionary computations, and so on. \cite{hamadahayanoichikikabatateramoto}, \cite{ichikihamada}, \cite{ichikihamada2}, \cite{sakuraisaekicarrwuyamamotoduketakahashi}, and so on, are closely related studies and the author has been interested in this field. 
For example, Naoki Hamada is an expert of multi-optimization problems together with related topics and kindly proposed a strategy of studying multi-functions for the problems via topological theory of their Reeb spaces. 
More precisely, as one interesting application, he proposed an idea that complexity of the topology of the Reeb space can measure how difficult and complicated a multi-optimization problem is.
This together with backgrounds on geometry has motivated the author to study global topologies of Reeb spaces of smooth maps of important classes further.

The author declares that all data supporting the present study are in the present paper.

\end{document}